\documentclass[10pt,reqno]{amsart}
\pdfoutput=1

\usepackage[numbers]{natbib}
\usepackage[utf8]{inputenc}
\usepackage{amsmath}
\usepackage{amsfonts}
\usepackage{amssymb}
\usepackage{mathrsfs}
\usepackage{mathtools}
\usepackage{enumitem}
\usepackage{url}
\usepackage{hyperref}
\usepackage[noabbrev,capitalize,nameinlink]{cleveref}

\usepackage{amsthm}
\newtheorem{thm}{Theorem}[section]
\newtheorem{lem}[thm]{Lemma}
\newtheorem{prop}[thm]{Proposition}
\newtheorem{cor}[thm]{Corollary}

\theoremstyle{definition}

\theoremstyle{remark}
\newtheorem*{rem}{Remark}

\newcommand\abs[1]{\left|#1\right|}

\newcommand\loc{\mathrm{loc}}

\DeclareMathOperator{\Div}{div}

\numberwithin{equation}{section}

\crefformat{equation}{(#2#1#3)}

\begin{document}

\title[Rotational Symmetry]{Rotational Symmetry of Solutions of Mean Curvature Flows Coming Out of A Double Cone}

\author{Letian Chen}
\address{Department of Mathematics, Johns Hopkins University, 3400 N. Charles Street, Baltimore, MD 21218}
\email{lchen155@jhu.edu}

\date{Jan. 24, 2021}

\begin{abstract}
We show that any smooth solution to the mean curvature flow equations coming out of a rotationally symmetric double cone is also rotationally symmetric.
\end{abstract}

\maketitle

\section{Introduction}
We say a family of properly embedded smooth hypersurface $\{\Sigma_t\}_{t \in I} \subset \mathbb{R}^{n+1}$ is a solution of the mean curvature flow (MCF) equations if
\begin{align*}
    \left(\frac{\partial x}{\partial t}\right)^\perp = H_{\Sigma_t}(x).
\end{align*} 
Here $H_{\Sigma_t}(x)$ denotes the mean curvature vector of $\Sigma_t$ at $x$, and $x^\perp$ is the normal component of $x$. \par
In this article we are interested in solutions of MCF coming out of a rotationally symmetric double cone, by which we mean a (hyper)cone $\mathcal{C} \subset \mathbb{R}^{n+1}$ whose link $\mathcal{L}(\mathcal{C}) = \mathcal{C} \cap \mathbb{S}^{n}$ is a smooth hypersurface of $\mathbb{S}^{n}$ and has two connected components lying in two separate hemispheres. More explicitly, we consider a cone of the form (up to an ambient rotation so that the axis of symmetry is $x_1$-axis)
\begin{equation}
\label{cone}
    x_1^2 =  \begin{cases} m_1(x_2^2 + x_3^2 + \cdots + x_n^2) & x_1 \ge 0 \\  m_2(x_2^2 + x_3^2 + \cdots + x_n^2) & x_1 < 0\end{cases}
\end{equation}
where $m_1,m_2 > 0$ are constants related to the aperture of the one. Solutions coming out of cones arise naturally in the singularity analysis of MCF. In particular the self-expanders, which are special solutions of the MCF satisfying $\Sigma_t = \sqrt{t}\Sigma$ for some hypersurface $\Sigma \subset \mathbb{R}^{n+1}$, are often thought of as models of MCF flowing out of a conical singularity (see for example \cite{AIC}). Self-expanders satisfy the elliptic equation
\begin{align*}
    H_\Sigma(x) = \frac{x^\perp}{2},
\end{align*}
which is the Euler-Lagrange equation of the functional $\int_\Sigma e^{\abs{x}^2/4} d\mathcal{H}^n$. We can therefore talk about the Morse index of a given self-expander, and the Morse flow lines between two self-expanders (asymptotic to the same cone $\mathcal{C}$) are examples of non self-similar solutions coming out of the cone.  \par 
We show that, given a smooth double cone $\mathcal{C} \subset \mathbb{R}^{n+1}$ and a solution to the MCF, $\{\Sigma_t\}_{t \in [0,T]}$, asymptotic to $\mathcal{C}$, then the flow inherits the rotational symmetry of $\mathcal{C}$ at all times. More precisely we prove:
\begin{thm}
\label{t11}
Let $\mathcal{C} \subset \mathbb{R}^{n+1}$ be a smooth, rotationally symmetric double cone. Suppose $\{\Sigma_t\}_{t \in [0,T]}$ is a smooth solution to the mean curvature flow asymptotic to $\mathcal{C}$, in the sense that
\begin{align*}
    \lim_{t \to 0^+} \mathcal{H}^n \llcorner \Sigma_t = \mathcal{H}^n \llcorner \mathcal{C}
\end{align*}
as Radon measures, then $\Sigma_t$ is also rotationally symmetric (with the same axis of symmetry) for any $t \in [0,T]$.
\end{thm}
\begin{rem}
It is likely that only a finite number of such solutions exist. These include self-expanders and Morse flow lines between two self-expanders asymptotic to the same cone, some of which can be constructed using methods from \cite{BW4}. In particular the latter solutions might develop singularities. Indeed, when the parameters $m_1$ and $m_2$ in \cref{cone} are sufficiently small, by \cite{Hel} we can find an unstable (connected) catenoidal self-expander and a disconnected self-expander whose two components are given by the unique self-expanders asymptotic to the top part and bottom part of the cone. One expects that there exists a Morse flow line connecting these two self-expanders. Such a flow line will necessarily develop a neck pinch in order to become disconnected.
\end{rem}
As an easy corollary we obtain the following rotational symmetry result:
\begin{cor}
\label{c12}
Let $\mathcal{C} \subset \mathbb{R}^{n+1}$ be a smooth, rotationally symmetric double cone, then any smooth self-expander $\Sigma$ asymptotic to $\mathcal{C}$ is also rotationally symmetric (with the same axis of symmetry).
\end{cor}
\begin{rem}
It is expected that no singular self-expander asymptotic to $\mathcal{C}$ exists, but our theorem only applies in the smooth case. The smoothness assumption is in place to avoid further technicality introduced by the moving plane method, see Section 2.4.
\end{rem}
The rotational symmetry is known in many other cases. Fong and McGrath \cite{FM} showed that same conclusion holds if the cone is rotationally symmetric and the expander is mean convex. Bernstein-Wang (Lemma 8.3 in \cite{BW3}) later showed that same conclusion holds if the cone is rotationally symmetric and the expander is weakly stable (in particular, mean convexity implies weak stability so this generalizes the Fong–McGrath result). In contrast, our result applies to all solutions coming out of the cone and does not assume any extra condition about the flow other than smoothness. For other geometric flows, Chodosh \cite{chod} proved rotational symmetry of expanding asymptotically conical Ricci solitons with positive sectional curvature. \par
It is also worth mentioning that, although in general given a rotationally symmetric smooth cone $\mathcal{C}$ there could be multiple self-expanders asymptotic to $\mathcal{C}$ , if there exists a unique self-expander asymptotic to $\mathcal{C}$, it must inherit the rotational symmetry. Uniqueness holds, for example, when the link of $\mathcal{C}$, $\mathcal{L}(C)$, is connected, or, in the double cone case, when the parameters $m_1,m_2$ in \cref{cone} are sufficiently large \cite{BW3}. It is interesting to determine whether the rotational symmetry holds when the link $\mathcal{L}(C)$ has 3 or more connected components. We suspect that counterexamples exist. We refer to \cite{BW0}, \cite{BW3}, \cite{BW1}, \cite{BW2}, and \cite{Ding} for more information on self-expanders. \par 
The proof of \cref{t11} relies on the moving plane method pioneered by Alexandrov to prove that embedded compact constant mean curvature hypersurfaces are round spheres. The method was further employed to minimal surfaces by Schoen \cite{Schoen} to prove certain uniqueness theorems for catenoids. More recently, Martín–Savas-Halilaj–Smoczyk \cite{MSS} showed uniqueness of translators (that is, solutions of the MCF equation that evolve by translating along one fixed direction) with one asymptotically paraboloidal end. Choi–Haslhofer–Hershkovitz \cite{CHH} and Choi–Haslhofer–Hershkovitz–White \cite{CHHW} used a parabolic variant of the method to deduce rotational symmetry of certain ancient solutions to the MCF equation (that is, solutions of the MCF equation which exist on $(-\infty,0)$). These methods were further generalized to non-smooth settings very recently by Haslhofer–Hershkovitz–White \cite{HHW} and by Bernstein–Maggi \cite{BM}. \par 
Although a self-expander $\Sigma$ satisfies an elliptic PDE, the hypersurface obtained after reflecting a self-expander with respect to a hyperplane does not satisfy the above equation anymore (it is rather a translated self-expander). For this reason we could not directly apply the usual elliptic maximum principle and Hopf lemma, and we need to work in spacetime $\mathbb{R}^{n+1} \times [0,T]$ and use the MCF equations directly with a parabolic version of the maximum principles, which will lead to the more general \cref{t11}. Consequently, our method is in spirit closer to that used by \cite{CHH}. 

\subsection*{Acknowledgment} The author would like to thank his advisor, Jacob Bernstein, for numerous helpful advice and constant encouragement, especially during a period of extreme difficulty around the globe. The author would also like to thank Rory Martin-Hagemeyer for useful discussions.

\section{Preliminaries}
\subsection{Notations} Throughout the paper, $B_r(x)$ will denote the Euclidean ball of radius $r$ centered at a point $x \in \mathbb{R}^{n+1}$. By a (smooth) MCF in $\mathbb{R}^{n+1}$ we mean a family of embedded hypersurfaces $\{\Sigma_t\}_{t \in I}$ for some interval $I$ such that
\begin{align*}
    \left(\frac{\partial x}{\partial t}\right)^\perp = H_{\Sigma_t}(x)
\end{align*}
for all $x \in \Sigma_t$, $t \in I$. Given an open set $U \subset \mathbb{R}^{n+1}$, we say $\{\Sigma_t\}_{t \in I}$ is a MCF in $U$ if the above equation is satisfied locally (given a local parametrization of the hypersurface) at every $x \in \Sigma_t \cap U$ and $t \in I$.
\subsection{Pseudolocality for MCF}
We will be frequently using the following pseudolocality result of Ilmanen-Neves-Schulze (see also \cite{HE}):
\begin{thm}[Theorem 1.5 of \cite{INS}]
\label{pseudo}
Let $\{\Sigma_t\}_{t \in (0,T]}$ be a mean curvature flow in $\mathbb{R}^{n+1}$. Given any $\eta > 0$, there is $\delta, \varepsilon > 0$ such that: if $x \in \Sigma_0$ and $\Sigma_0 \cap C_1(x)$ is a graph over $C^n_1(x)$ with Lipschitz constant bounded by $\varepsilon$, then $\Sigma_t \cap C_\delta(x)$ can be written as a graph over $C^n_\delta(x)$ with Lipschitz bounded by $\eta$ for any $t \in [0,\delta^2) \cap [0,T)$.
\end{thm}
Here \begin{align*}
    C_r^n(x) = \{(y,y_{n+1}) \in \mathbb{R}^{n+1} \mid \abs{y - x} < r\}
\end{align*}
for $(x,x_{n+1}) \in \mathbb{R}^{n+1}$ is the open cylinder in $\mathbb{R}^{n+1}$ centered at $x$ and 
\begin{align*}
    C_r(x) = \{(y,y_{n+1}) \in C_r^n(x) \mid \abs{y_{n+1} - x_{n+1}} < r\}
\end{align*}
is the closed cylinder. Roughly speaking, this theorem says that if the initial data of our MCF is graphical in some cylinder centered at $x$, then at least for a short time the evolution of the hypersurface stays graphical in a possibly smaller cylinder. We will primarily use this theorem to show that our flow is graphical outside of a large ball for a short time, although strictly speaking we sometimes need to apply the above theorem in the context of integral Brakke flow.
\subsection{Parabolic Maximum Principles}
In this section $Z_r(x,t)$ will denote the spacetime cylinder of radius $r$ centered at $(x,t) \in \mathbb{R}^{n} \times \mathbb{R}$; that is,
\begin{align*}
    Z_r(x,t) = \{(y,s) \in \mathbb{R}^{n} \times \mathbb{R} \mid \abs{y-x} < r, \abs{t-s} < r^{2}\}.
\end{align*}
$Z_r^-(x,t)$ will denote the part of cylinder $Z_r(x,t)$ whose time component is smaller than $t$. To carry out the moving plane method, the most important ingredients are the maximum principle and Hopf lemma. In our case we need a version of those theorems applicable to graphical solutions of MCF; that is, functions $u: Z_r^-(0,0) \to \mathbb{R}$ satisfying the following parametrized PDE:
\begin{align*}
    u_t = \sqrt{1 + \abs{\nabla u}^2} \Div\left(\frac{\nabla u}{\sqrt{1+\abs{\nabla u}^2}}\right).
\end{align*}
Observe that the difference of two graphical solutions to MCF satisfies a second-order linear parabolic PDE (provided the gradients are bounded a priori, which will be the case since our solutions are asymptotically conical), so by standard theory of linear parabolic PDEs \cite{Lieb} we have (cf. Section 6.2 in \cite{CHH}):
\begin{lem}[Maximum Principle]
\label{l22}
Suppose $u,v$ are graphical solutions to the MCF in a parabolic cylinder $Z_r^-(0,0)$ with $u(0,0) = v(0,0)$. If $u \le v$ in $Z_r^-(0,0)$, then $u = v$ in $Z_r(0,0)$.
\end{lem}
\begin{lem}[Hopf Lemma]
\label{l23}
Suppose $u,v$ are graphical solutions to the MCF in a half parabolic cylinder $Z_r^-(0,0) \cap \{x_{1} \ge 0\}$ with $u(0,0) = v(0,0)$ and $\frac{\partial u}{\partial x_1}(0,0) = \frac{\partial v}{\partial x_1}(0,0)$. If $u \le v$ in $Z_r^-(0,0) \cap \{x_1 \ge 0\}$, then $u = v$ in $Z_r^-(0,0) \cap \{x_1 \ge 0\}$.
\end{lem}
\subsection{Asymptotically Conical Mean Curvature Flow}
Here we will briefly discuss the class of MCFs we consider in \cref{t11}. We need our MCF to be at least $C^{2,\alpha}$-asymptotically conical in order to apply the maximum principle and Hopf Lemma. This will almost not be an issue if we assume our cone is at least $C^3$, as the following proposition shows. 
\begin{prop}[cf. Proposition 3.3 in \cite{BW1}]
\label{p24}
Let $\mathcal{C}$ be a $C^3$ cone and suppose $\{\Sigma_t\}_{t \in (0,T]}$ is a MCF such that
\begin{align*}
    \lim_{t \to 0^+} \mathcal{H}^n \llcorner \Sigma_t = \mathcal{H}^n \llcorner \mathcal{C},
\end{align*}
then we have for $\alpha \in [0,1)$ and $t \in (0,T)$,
\begin{align*}
    \lim_{\rho \to 0^+} \rho \Sigma_t =\mathcal{C} \text{ in } C^{2,\alpha}_{loc}(\mathbb{R}^{n+1} \setminus \{0\}).
\end{align*}
\end{prop}
\begin{proof}
It is enough to prove that locally $\Sigma_t$ is a $C^{2,\alpha}$ normal graph over $\mathcal{C}$ outside of a large ball. By \cref{pseudo} (strictly speaking we need to consider the flow together with the initial condition $\Sigma_0 = \mathcal{C}$ as an integral Brakke flow and apply the theorem for Brakke flows), there is $\delta > 0$ such that $\Sigma_t \cap C_\delta(x_0)$ can be written as a normal graph over $C_\delta^n(x)$ with Lipschitz constant bounded by $\eta$ for $t \in (0,\delta^2)$. This induces a map $u_{x_0} :[0,\delta^2) \times C_\delta^n(x_0) \to \mathbb{R}$ whose graph describes part of the flow in the space-time of the flow:
\begin{align*}
    \mathcal{M} = \mathcal{C} \times \{0\} \cup \bigcup_{t \in (0,T)} \Sigma_t \times \{t\}.
\end{align*}
Since $\{\Sigma_t\}$ is a MCF, $u_x$ satisfies the parametrized equation
\begin{align*}
    \frac{\partial u_{x_0}}{\partial t} = \sqrt{1 + \abs{\nabla_x u_{x_0}}^2} \Div\left(\frac{\nabla_x u_{x_0}}{\sqrt{1+ \abs{\nabla_x u_{x_0}}^2}}\right)
\end{align*}
with initial conditions $u_{x_0}(0,x_0) = \abs{\nabla_x u_{x_0}(0,x_0)} = 0$. Since we did not introduce the various H\"{o}lder and weighted H\"{o}lder norms (to account for the cone structure), we will only briefly outline the proof, which is essentially the same as the proof of Proposition 3.3 in \cite{BW1}. This is a quasilinear parabolic PDE in divergence form and we may use Schauder estimates (Theorem 5.1 in Chapter 5 of \cite{LSU}) to get
\begin{align*}
    \abs{\nabla_x u_{x_0}(t,x)} \le C(\abs{x-x_0} + \sqrt{t}) \le C\delta.
\end{align*}
For the pointwise estimate on $u_{x_0}(t,x)$ we have to combine the above and the MCF equation to bound the time derivative $\abs{\partial_t u_{x_0}(t,x)}$, giving an estimate of the form
\begin{align*}
     \abs{u_{x_0}(t,x)} \le C(\abs{x-x_0}^2 + t) \le C\delta^2.
\end{align*}
Finally, Schauder estimates (Theorem 1.1 in Chapter 6 of \cite{LSU}) gives Holder semi-norm bounds for any $\alpha \in (0,1)$. Combining these estimates yield a weighted-$C^{2,\alpha}$ estimate on the cone, which is what we needed. 
\end{proof}
Unfortunately pseudolocality only gives normal graphicality outside of a large compact set, and so we can not conclude that the entire flow will be of class $C^{2,\alpha}$. For this reason it is assumed that the MCF is smooth to begin with in \cref{t11}. We note that it might be possible to remove this assumption using a moving plane method in non-smooth settings such as those presented in \cite{CHHW}, \cite{BM} or \cite{HHW}.
\section{Rotational Symmetry}

In this section we prove \cref{t11}. As claimed before, a direct consequence of the pseudolocality theorem \cref{pseudo} is the graphicality of the immortal solution outside of a large ball. For the next lemma we denote $\Sigma^+ = \Sigma \cap \{x_{n+1} > 0\}$ and $\Sigma^- = \Sigma \cap \{x_{n+1} < 0\}$ for $\Sigma \subset \mathbb{R}^{n+1}$.

\begin{figure}
\centering
\resizebox{2in}{!}{\input{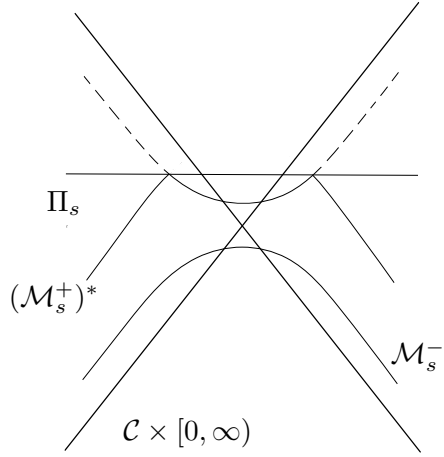}}
\caption{A typical picture of the moving plane}
\end{figure}

\begin{lem}
\label{l41}
Let $\mathcal{C}, \{\Sigma_t\}_{t \in (0,T)}$ be as in \cref{t11}. For each $t \in [0,T)$ there is $R = R(\mathcal{C},\Sigma,t)$ such that $(\Sigma_t)^+ \setminus B_R(0)$ is graphical over $\Pi_0 \setminus B_R(0)$, where $\Pi_0 = \{x_{n+1} = 0\}$; that is, the projection $\pi:(\Sigma_t)^+ \setminus B_R(0) \to \Pi_0$ is injective. The same holds for $(\Sigma_t)^-$.
\end{lem}
\begin{proof}
    Let $\eta > 0$. Consider the MCF $\{\Sigma_t\}_{t \in (0,T)}$ together with the initial data $\Sigma_0 = \mathcal{C}$ and apply \cref{pseudo}. Since the cone is smooth this can be treated as an integral Brakke flow (with no sudden mass loss). Take any point $x \in \mathcal{C}^+ \setminus B_1(0)$. Since $\mathcal{C}^+$ is clearly graphical over itself, by pseudolocality theorem \cref{pseudo} there exists $t_1 > 0$ such that $0 < t < t_1$ implies $(\Sigma_t)^+ \cap C_{\sqrt{t_1}}(x)$ is a normal graph over $B_{\sqrt{t_1}}(x)$ with Lipschitz constant bounded above by $\eta$. Since $x$ is arbitrary, this shows that $(\Sigma_t)^+ \setminus B_1(0)$ is a normal graph over $\mathcal{C}^+ \setminus B_1(0)$ for $0 < t < t_1$ with Lipschitz constant bounded above by $\eta$. For a general $t > t_1$ parabolic rescaling shows shows that for $R = tt_1^{-1}$ we have that $(\Sigma_t)^+ \setminus B_R(0)$ is a normal graph over $(\mathcal{C}^+) \setminus B_R(0)$ with Lipschitz constant bounded above by $\eta$. \par 
    The cone $\mathcal{C}^+$ is graphical over $\Pi_0$ with a fixed angle $\theta < \frac{\pi}{2}$. Moreover we have proved that given $\eta > 0$ there is $R >0$ such that $(\Sigma_t)^+ \setminus B_R(0)$ is a normal graph over $\mathcal{C}^+ \setminus B_R(0)$, so by choosing $\eta$ small enough we have that the angle between the tangent planes of $\mathcal{C}^+ \setminus B_R(0)$ and $(\Sigma_t)^+ \setminus B_R(0)$ are as small as we want. Hence $(\Sigma_t)^+ \setminus B_R(0)$ is graphical over $\Pi_0$ as well.
\end{proof}

For the rest of the section, let 
\begin{align*}
    \Pi_s = \{(x,x_{n+1}) \in \mathbb{R}^{n+1} \mid x_{n+1} =  s\} \times [0,\infty) \subset \mathbb{R}^{n+1} \times [0,\infty)
\end{align*}
be the hyperplane at level $s$ in spacetime.  Given a set $A \subset \mathbb{R}^{n+1} \times [0,\infty)$ and $t,s \in [0,\infty)$ we let 
\begin{align*}
    A^t = \{(x,x_{n+1},t') \in A \mid t' = t\}
\end{align*}
be the time $t$ slice of $A$,
\begin{align*}
    A_s^+ = \{(x,x_{n+1},t) \in A \mid x_{n+1} >  s\}
\end{align*}
be the part of $A$ lying above $\Pi_s$, $A_s^-$ be the part of $A$ lying below $\Pi_s$ and finally 
\begin{align*}
    A^*_s = \{(x,x_{n+1},t) \mid (x, 2s - x_{n+1},t) \in A\}
\end{align*}
be the reflection of $A$ across $\Pi_s$, but we will often drop the subscript $s$ when it is understood to avoid excessive subscripts. Given two sets $A,B \subset \mathbb{R}^{n+1} \times [0,\infty)$ we say $A > B$ if for any $(x,x_{n+1},t) \in A$ we have $x_{n+1} > y_{n+1}$ for any $(x,y_{n+1},t) \in B$ (if there is any such point). 
\begin{proof}[Proof of \cref{t11}]
    Without loss of generality assume that $\mathcal{C}$'s axis of symmetry is the $x_1$-axis. Evidently it suffices to show that the flow preserves the reflection symmetry across any hyperplane containing the $x_1$-axis, which without loss of generality we will take to be $\{x_{n+1} = 0\}$. \par 
    We will use the moving plane method on the spacetime $\mathbb{R}^{n+1} \times [0,\infty)$. The spacetime track
    \begin{align*}
        \mathcal{M} = \bigcup_{t \in [0,T)} \Sigma_t  \times \{t\}
    \end{align*}
    is a properly embedded hypersurface in $\mathbb{R}^{n+1} \times [0,T)$ asymptotic to $\mathcal{C} \times [0,T)$, in the sense that at each time slice $t$, $\Sigma_t$ is $C^{2,\alpha}$-asymptotic to $\mathcal{C}$ as we have demonstrated in \cref{p24}. Let
    \begin{align*}
        S = \{s \in [0,\infty) \mid (\mathcal{M}_s^+)^* > \mathcal{M}_s^-,  (\mathcal{M}_s^+)^t \text{ is graphical over $(\Pi_s)^t$ for $t \in [0,T]$} \}.
    \end{align*}
    Here by graphical we meant that $(\mathcal{M}_s^+)^t$ can be written as a normal graph over $(\Pi_s)^t$. Alternatively, since our solution is smooth we can require that the vertical vector $e_{n+1} = (0,\ldots,0,1)$ is not contained in the tangent space of any point $p \in (\mathcal{M}_s^+)^t$. We first note that since the cone is symmetric across $\Pi_0$ we have $(\mathcal{C}_s^+)^* > \mathcal{C}_s^-$ for every $s>0$. It is not hard to see that $S$ is an open set. In fact we just need to show that $e_{n+1}$ is not in the tangent space at infinity for $(\mathcal{M}_s^+)^t$. By \cref{p24},
    \begin{align*}
        \lim_{\rho \to 0^+} \rho (\mathcal{M}_s^+)^t = \mathcal{C}
    \end{align*}
    in $C^{2,\alpha}_{\loc}(\mathbb{R}^{n+1} \setminus \{0\})$, so eventually the tangent space at a point $p \in (\mathcal{M}_s^+)^t$ will lie close to the tangent space of $\mathcal{C}$. Since the cone is not vertical $e_{n+1}$ is clearly not contained in the tangent space of any point on $\mathcal{C}$, so by the convergence there is $\varepsilon > 0$ such that $e_{n+1}$ is not in the tangent space of any point $p \in (\mathcal{M}_{s-\varepsilon}^+)^t$.  \par
    
    \begin{figure}
    \centering
    \begin{minipage}{0.48\textwidth}
    \centering
    \resizebox{2in}{!}{\input{img1.tex}}
    \caption{Boundary touching}
    \end{minipage}\hfill
    \begin{minipage}{0.48\textwidth}
    \centering
    \resizebox{2in}{!}{\input{img2.tex}}
    \caption{Interior touching}
    \end{minipage}

    \end{figure}

    By \cref{l41}, there is $R > 0$ such that $\Sigma_1 \setminus B_R(0) = \mathcal{M}^1 \setminus B_R(0)$ is graphical over $(\Pi_0)^1$. Moreover this graphical scale scales parabolically, so for $s > T^2R$ we have $(\mathcal{M}_s^+)^t$ is graphical over $(\Pi_s)^t$ for each $t \in [0,T]$. It is also evident that $(\mathcal{M}_s^+)^t$ is asymptotic to the translated cone $\mathcal{C}+2se_{n+1}$, so when $s$ is large enough the reflected part is disjoint from $(\mathcal{M}_s^-)_0^+$ (that is, the part of $\mathcal{M}$ that lies below level $s$ and above $0$). Together with graphicality this implies $(\mathcal{M}_s^+)^* > \mathcal{M}_s^-$ for sufficiently large $s$, so $S$ is not empty. \par 
    Finally we show $S$ is closed. Suppose for a contradiction that $(s,\infty) \subset S$ (clearly $s \in S$ implies $[s,\infty) \subset S$) but $s \not \in S$. At level $s$, either the graphicality condition or the set comparison condition $(\mathcal{M}_s)^* > \mathcal{M}_s^-$ is violated. In the first case, by parabolic rescaling we may assume for simplicity that the nongraphicality happens first at time $t = 1$. This means that there is $p \in (\mathcal{M}_s^+)^1$ such that $e_{n+1} \in T_p(M_s^+)^1$. Thus tangent planes of $(\mathcal{M}_s^+)^*$ and $\mathcal{M}_s^-$ at the point $(p,1)$ must coincide. If we choose $r$ small enough we can ensure that $(\mathcal{M}_s^+)^*$ and $\mathcal{M}_s^-$ are graphical over $Z_r^-(p,1) \cap \{x_{n+1} \le s\}$. Since the tangent planes coincide we can apply Hopf Lemma \cref{l23} to $(\mathcal{M}_s^+)^*$ and $\mathcal{M}_s^-$ to conclude that these hypersurfaces agree on an open neighborhood of $(p,1)$. Moreover, the set $(\mathcal{M}_s^+)^* \cap \mathcal{M}_s^-$ is closed by definition and open by the maximum principle, so at least a connected component of $(\mathcal{M}_s^+)^*$ must coincide with a component of $\mathcal{M}_s^-$. This implies that $(\mathcal{M}_s^+)^*$ is asymptotic to both the cones $\mathcal{C} \times [0,\infty)$ and $(\mathcal{C}+2se_{n+1}) \times [0,\infty)$, a contradiction. In the second case, $s$ is necessarily the first level such that $(\mathcal{M}_s)^+ \cap \mathcal{M}_s^- \ne \emptyset$, and the graphicality condition implies that $(\mathcal{M}_s^+)^*$ and $\mathcal{M}_s^-$ must touch at an interior point $(p,t)$ of the flow. Again for $r$ small enough they are both graphical solutions of the MCF, so the maximum principle \cref{l22} implies that $(\mathcal{M}_s^+)^*$ and $\mathcal{M}_s^-$ agree on an open neighborhood of $(p,t)$. Since we can do this for any point $(p,t) \in (\mathcal{M}_s^+)^* \cap \mathcal{M}_s^-$, at least a connected component of $(\mathcal{M}_s^+)^*$ coincides with a component of $\mathcal{M}_s^-$, a contradiction. \par 
    We have thus proved that $S$ is open, non-empty and closed, and thus $S = (0,\infty)$ (note that since we have a strict inequality in our set up so we cannot conclude directly that $0 \in S$). Note that we can run a similar argument starting from the bottom half, yielding $(\mathcal{M}_s^-)^* > \mathcal{M}_s^+$ for any $s < 0$. Hence there must be a point of touching at $s = 0$. If the intersection is in the interior we can apply the maximum principle \cref{l22} to conclude that $(\mathcal{M}_0^+)^* = \mathcal{M}_0^-$, i.e. $\mathcal{M}$ is symmetric across the reflection with respect to $\Pi_0$. The same conclusion holds if the intersection is along the boundary by using the Hopf lemma \cref{l23} instead.
\end{proof}
\begin{rem}
	In the proof of \cref{t11} we actually proved the stronger result that reflection symmetry is preserved for self-expanders asymptotic to a double cone. This yields, for example, that when $m_1 = m_2$ in \cref{cone}, any self-expander asymptotic to $\mathcal{C}$ is symmetric across the reflection with respect to the plane $\{x_1 = 0\}$. We do not expect this to hold for cones whose links have more than 2 components. 
\end{rem}

\end{document}